\documentclass[11pt]{amsart}

\usepackage{geometry}
\usepackage{mathtools}
\usepackage{amsmath,amssymb,color, array, verbatim}
\usepackage{enumerate}
\usepackage{verbatim}

\newtheorem{thm}{Theorem}

\newtheorem{lem}[thm]{Lemma}

\newtheorem{prop}[thm]{Proposition}

\theoremstyle{definition}

\newtheorem{rem}[thm]{Remark}

\numberwithin{equation}{section}

\newcommand{\Z}{{\mathbb Z}}

\newcommand{\Tr}{{\mathrm{Tr}}}
\newcommand{\diag}{\mathrm{diag}}

\newcommand{\F}{\mathbb{F}}
\newcommand{\M}{\mathbb{M}}
\newcommand{\rvline}{\hspace*{-\arraycolsep}\vline\hspace*{-\arraycolsep}}

\begin{document}

\title{$m$-nil-clean nonderogatory matrices}

\author[A. Pojar]{Andrada Pojar}
\address{Technical University of Cluj-Napoca, Department of Mathematics, Str. Memorandumului 28, 400114, Cluj-Napoca, Romania}
\email{andrada.pojar@math.utcluj.ro}

\begin{abstract}
We prove that if $\F$ is a field of positive odd characteristic $p,$ and $m,$ and $n$ are positive integers such that $m\geq2,$ and $n\leq p,$ every $n\times n$ nonderogatory matrix $A\in \M_n(\F)$ which is sum of $m$ idempotents and a nilpotent, has a decomposition $A=E_1+E_2+\dots+E_m+V,$ such that $E_i^2=E_i,$ for every $i\in \{1,\dots,m\},$ and $V^{[\frac{p-2}{m}]+2}=0.$
\end{abstract}

\keywords{$m$-nil-clean decomposition; nonderogatory matrix; companion matrix; nilpotent matrix; idempotent matrix }

\subjclass[2010]{15A24, 15A83, 16U99} 

\maketitle

\section{Introduction}

Let $\F$ be a field with positive odd characteristic $p.$
As usual, the letter $\F_p=\mathbb{Z}_p$ will stand for the prime field of $p$ elements having characteristic $p$, for any positive integer $n$, the notation $\mathbb{M}_n(\F)$ will denote the full matrix ring of $n\times n$ matrices over $\F,$
 and $\mathbb{M}_{m,n}(\F)$ will denote the full matrix ring of  matrices over $\F,$ with $m$ rows and $n$ columns.

A square matrix is {\it nonderogatory} if its characteristic
and minimal polynomials coincide. A matrix is nonderogatory if and only if it is similar to a
companion matrix $C$.
We recall that a companion matrix $C\in Mn(\F)$ is a matrix of the form
$$C=C_{c_0,c_1,\ldots, c_{n-1}}=\left(\begin{array}{ccccc}
0 & 0  &\ldots & 0 & -c_0 \\
1 & 0  &\ldots & 0 & -c_1\\
\vdots & \vdots  &\cdots  & \vdots & \vdots \\
0 & 0 &\ldots & 1 & -c_{n-1}
 \end{array}\right).$$

 A square matrix $A$ over $\F$ is nil-clean if there exist an idempotent $E$ and a nilpotent $N$ such that $A=E+N.$ We consider decompositions of nonderogatory matrices $A,$ such that there exist idempotents $E_i$, $i\in \{1,2,\dots,m\},$ and a nilpotent $N,$ with $A=E_1+E_2\dots+E_m+N,$ which are $m$-nil-clean decompositions.

Nil-clean decompositions are related to clean ones, introduced by Nicholson in \cite{N}, when investigating exchange rings and were first studied by Diesl in \cite{Diesl}. An important result appeared in \cite{BCDM} and \cite{KLZ} about them is: every $n\times n$ matrix over a division ring $D$ is nil-clean if and only if $D=\F_2.$

A theme related to a set $S$ of nilpotents of a ring is the study of boundness of the nilpotence index of S--the existence of a positive integer $n$ such that $a^n=0$ for all $a\in S$ , and it was first studied first in \cite{KWZ}. \v{S}ter has proved that there exists a nil-clean decomposition of every matrix of $\mathbb{M}_n(\F_2),$ with nilpotent $N$ having nilpotence index at most $4$ (i.e $N^4=0.$). Moreover this result has sharpness, that is there exist $4\times 4$ companion matrices over $\F_2,$ which cannot be decomposed as a sum of an idempotent and a nilpotent matrix of nilpotence index at most $3.$ Also in \cite{Pojar-nonderog-p-pot-nil} it has been proved that every nonderogatory matrix $A$ over a field of positive odd characteristic $p,$ that is sum of a $p$-potent matrix and a nilpotent matrix, has a decomposition $A=E+V,$ such that $E^p=E,$ and $V^3=0.$ Moreover, this result has sharpness, that is there exists a $3\times 3$ companion matrix over a field of characteristic $3,$ which is sum of a tripotent and a nilpotent, but cannot be decomposed as a sum of a tripotent and a nilpotent matrix of nilpotence index at most $2.$

An extension of nil-clean decompositions to finite fields of odd cardinality $q,$ was done in \cite{AM}: every matrix over such a field is the sum of a $q$-potent $E=E^q,$ and a nilpotent. In \cite{B} there is even more -- the nilpotent $N$ involved in such a decomposition can be with nilpotence index at most $3$ (i.e. $N^3=0).$

In \cite{BrMe} it has been proved that every nil-clean nonderogatory $n\times n$ matrix $A,$ over a field of positive characteristic, $p,$ with trace of $A$ not equal to $1$ has a decomposition $A=E+V,$ such that $E^2=E,$ and $V^{p+1}=0.$ If $\mbox{trace}(A)=1,$ then there is a similar decomposition with $V^{p+2}=0.$

We study decompositions of nonderogatory matrices that are sums of $m$ idempotents and a nilpotent. By Lemma 3.1 from \cite{C} we know that the trace of such a matrix is an integer multiple of unity of $\F.$ If the characteristic $p$ of the field is less then the size $n$ of the matrix, then a companion matrix with trace an integer multiple of unity is just nil-clean, so we can assume $n\leq p.$ The following theorem provides a characterization for $m$-nil-clean companion matrices.
\begin{thm}\label{teoremaMNilClean1}\cite{C}
Let $m\geq 2$ be an integer. Let $\mathbb{F}$ be a field of positive odd characteristic $p,$   and $1<n\leq p$. Let $C=C_{c_0,c_1,\ldots, c_{n-1}}$ be a  companion matrix.
Let $c$ be the nonnegative integer such that $-c_{n-1}=c\cdot 1.$

The following hold:
\begin{enumerate}
\item   if $c=0$ and $mn-1<p$ then $C$ is  $m-$nil-clean if and only if $C$ is nilpotent or $C-(m-1)I_n$ is unipotent;
\item if $c=t$, $1\leq t\leq m$ then $C$ is $t-$nil-clean ($1-$nil-clean is just nil-clean), so is $m-$nil-clean;
\item if $c\in \{m,m+1,\ldots, mn-2,mn-1\}$ then $C$ is $m-$nil-clean,
\item if $mn-2\geq p$ then $C$ is $m-$nil-clean;
\item Assume $mn-2<p,$
      \begin{enumerate}
      \item if $c=mn$ and $p=mn-1$ then $C$ is nil-clean, so is $m-$nil-clean;
      \item if $c=mn$ and $p=mn$ then $C$ is $m-$nil-clean if and only if $C$ is nilpotent or $C-(m-1)I_n$ is a unipotent matrix;
      \item if $c=mn$, $p>mn$  then $C$ is $m-$nil-clean if and only if $C-(m-1)I_n$ is a unipotent matrix.
      \item if  $c>mn$ then $C$ is not $m-$nil-clean.
      \end{enumerate}
\end{enumerate}
\end{thm}
By the assumption $mn-1\geq p$ every nonderogatory matrix with trace an integer multiple of unity of $\F$ is $m$-nil-clean, considering the fact that companion matrices with trace in $\{1,2,\dots,mn-1\}$ are $m$-nil-clean. We prove that if $mn-1\geq p\geq n,$ every nonderogatory matrix $A$ over a field of positive odd characteristic $p,$ that is sum of $m$ idempotents and a nilpotent, has a decomposition $A=E_1+E_2+\dots+E_m+V,$ such that $E_i^p=E_i,$ for $i\in \{1,\dots,m\},$ and $V^{[\frac{p-2}{m}]+2}=0.$ Moreover this result has sharpness, that is for $m=2$, and $n=[\frac{p-2}{m}]+2,$ there exists an $n\times n$ companion matrix, over a field of odd characteristic $p,$   which  are sum of two idempotents and a nilpotent, but cannot be decomposed as a sum of two idempotents and a nilpotent matrix of nilpotence index at most $[\frac{p-2}{m}]+1.$ If $mn-1\leq p-1$ then since $m\geq 2,$ then
  $$n\leq \frac{p}{m}\leq \frac{p-2}{m}+1<[\frac{p-2}{m}]+2.$$
  Therefore since the nilpotent $V$ involved in an $m$-nil-clean decomposition of $A$ is such that $V^n=0,$ we get $V^{[\frac{p-2}{m}]+2}=0.$

\section{Useful lemmas}

The following lemma provides a useful tool in the study of $m$-nil-clean decompositions for nonderogatory matrices.

\begin{lem}\label{use1}\cite{Pojar-comp}
 Let $k$ and $n$ be nonzero natural numbers, $n\geq k,$ and $\F$ be a field. Let $a_1,a_2,\dots,a_k$ be integer multiples of the unity. Every companion matrix $C\in M_n(\mathbb{F})$ is similar to the matrix  $C'+\diag (a_1,a_2,\dots,a_k,0,\dots,0),$ for some companion matrix $C'.$
\end{lem}

\begin{rem}\label{formP}
By Lemma \ref{use1} for the companion matrix $C\in \M_n(\F)$ there exist a companion matrix $C'$ and an invertible matrix $P$ such that $C=P(\diag (\underbrace{a_1,\dots,a_k }_{k \textrm{-times}} ,0,\dots,0)+C')P^{-1}$. This matrix $P$ is upper triangular with all entries on the main diagonal being equal to $1.$
\end{rem}

The following Lemma is the $m$-nil-clean corespondent of Lemma $2$ from \cite{BrMe}.

\begin{lem}\label{mainLem}
Let $n\geq 3$ be a positive integer, $\F$ a field, $m\geq 2$ an integer, and $k\in \{1,2,\ldots,n-1\},$
such that $n-k$ is odd.
Let $C\in \mathcal{M}_n(\F)$ be a companion matrix. If $\mbox{trace}(C)=(\frac{n+k+1}{2}\cdot m-l)\cdot 1,$ for some $l\in \{1,2,\dots,mk-1\},$
 then $C$ is similar to a matrix $D$ which has a decomposition $D=E_1+E_2+\dots+E_m+V,$ such that $E_i^2=E_i,$ $i\in\{1,2,\dots,m\},$ $V^{k+1}=0,$ the first column of $V$ is $(0,0,\ldots,0),$ and the last line of $E_m$ is $(0,0,\ldots,0,1).$
\end{lem}
\begin{proof}
Let $(e_1,e_2,\ldots,e_n)$ be the canonical basis of $\F^n$. We construct inductively a new basis $(f_1,f_2,\ldots,f_n)$ in the following way:
\\ $f_1=e_1;$
\\ for every even $i\in \{2,\ldots, n-k+1\},$ $f_i=Cf_{i-1}-mf_{i-1};$
\\ for every odd $i\in \{2,\ldots, n-k+1\},$ $f_i=Cf_{i-1};$
\\ for every $i\in \{n-k+2,\ldots,n\},$ $f_i=mf_{i-1}-Cf_{i-1};$

It can be proved that for every $j\in \{1,2,\ldots,n\}$  we have $\textrm{Lin}(\{e_1,\ldots,e_j\})=\textrm{Lin}(\{f_1,\ldots,f_j\}),$ hence $\{f_1,\ldots,f_n\}$ is a basis for $\F^n.$
The matrix associated to $C$ in the basis $\{f_1,\ldots,f_n\}$ has the form $$D=\left(\begin{array}{cc}
M & N \\
P & Q
\end{array}\right),$$
with $$M=\left(\begin{array}{cccccccc}
m  & 0 & 0 & 0 & \dots & 0 & 0 & 0 \\
1  & 0 & 0 & 0 & \dots & 0 & 0 & 0 \\
0  & 1 & m & 0 & \dots & 0 & 0 & 0 \\
0  & 0 & 1 & 0 & \dots & 0 & 0 & 0 \\
\vdots & \vdots & \vdots & \vdots & \ddots & \vdots & \vdots & \vdots \\
0  & 0 & 0 & 0 & \dots & m & 0 & 0 \\
0  & 0 & 0 & 0 & \dots & 1 & 0 & 0 \\
0  & 0 & 0 & 0 & \dots & 0 & 1 & m
\end{array}\right)\in \mathcal{M}_{n-k}(\F),$$

$$N=\left(\begin{array}{ccccc}
 0 & 0 & \dots & 0 & d_1 \\
 0 & 0 & \dots & 0 & d_2 \\
 \vdots & \vdots & \ddots & \vdots & \vdots \\
 0 & 0 & \dots & 0 & d_{n-k-1} \\
 0 & 0 & \dots & 0 & d_{n-k}
\end{array}\right)\in \mathcal{M}_{n-k,k}(\F),$$

$$P=\left(\begin{array}{ccccc}
 0 & 0 & \dots & 0 & 1 \\
 0 & 0 & \dots & 0 & 0 \\
 \vdots & \vdots & \ddots & \vdots & \vdots \\
 0 & 0 & \dots & 0 & 0 \\
 0 & 0 & \dots & 0 & 0
\end{array}\right)\in \mathcal{M}_{k,n-k}(\F),$$

and $Q=mI_k-C_{d_{n-k+1},\ldots,d_n}$ for some $d_{n-k+1},\ldots,d_n\in \F.$ Moreover from $$\mbox{trace}(M)+\mbox{trace}(Q)=\mbox{trace}(C),$$ it follows that
$$\frac{n-k+1}{2}\cdot m+km+d_{n}=(\frac{n+k+1}{2}\cdot m-l)\cdot 1,$$ hence
$$\frac{n+k+1}{2}\cdot m+d_{n}=(\frac{n+k+1}{2}\cdot m-l)\cdot 1,$$ therefore
$$-d_{n}=l\in \{1,2,\dots,mk-1\}.$$

By the proofs of Lemmas 3.2 and 3.3 from  \cite{C}, it follows that there exist $r_i\in \{1,\dots,n\},$ and the idempotents $F_i$ such that
$$C_{d_{n-k+1},\ldots,d_n}=F_1+\dots+F_{m-1}+B,$$
where $F_i=P_i\cdot \dots \cdot P_1\cdot\diag (\underbrace{1,\dots,1 }_{r_i \textrm{-times}} ,0,\dots,0)\cdot (P_i\cdot \dots \cdot P_1)^{-1},$ $B=P_{m-1}\cdot \dots \cdot P_1\cdot C'(P_{m-1}\cdot \dots \cdot P_1)^{-1},$
$C'$ is a $k\times k$ companion matrix, with $l_1=\Tr(C')\in\{1,\dots,k-1\},$ and $P_i$ is a matrix of the form of $P$ from Remark \ref{formP}.
Since $l_1=\Tr(C')\in\{1,\dots,k-1\},$ we derive by Lemma 1 from \cite{BrMe} that there exists $L_{12}$ and $T_1$ a nilpotent such that $C'=\left(\begin{array}{cc}
I_{l_1}  & L_{12} \\
0 & 0
\end{array}\right)+T_1.$
By Remark \ref{formP} all $P_i's $ are upper triangular with only $1's$ on the main diagonal, therefore there exists $L'_{12}$ such that
$$B=\left(\begin{array}{cc}
I_{l_1}  & L'_{12} \\
0 & 0
\end{array}\right)+(P_{m-1}\cdot \dots \cdot P_1)^{-1}\cdot T_1P_{m-1}\cdot \dots \cdot P_1.$$
Let $T$ be the nilpotent matrix $T=-(P_{m-1}\cdot \dots \cdot P_1)^{-1}\cdot T_1P_{m-1}\cdot \dots \cdot P_1$
We consider $$V=\left(\begin{array}{cc}
R  & S \\
0 & T
\end{array}\right),$$
with
$$R=\left(\begin{array}{ccccccc}
0  & 0 & 0 & 0 & \dots & 0 & 0   \\
0  & 0 & 0 & 0 & \dots & 0 & 0   \\
0  & 1 & 0 & 0 & \dots & 0 & 0   \\
0  & 0 & 0 & 0 & \dots & 0 & 0  \\
\vdots & \vdots & \vdots &  \vdots & \ddots & \vdots & \vdots \\
0  & 0 & 0 & 0 & \dots & 0 & 0 \\
0  & 0 & 0 & 0 & \dots & 1 & 0
\end{array}\right)\in \mathcal{M}_{n-k}(\F),$$
(where $0$ and $1$ alternate on the diagonal below the main one),

$$S=\left(\begin{array}{ccccc}
 0 & 0 & \dots & 0 & d_1 \\
 0 & 0 & \dots & 0 & 0 \\
 0 & 0 & \dots & 0 & d_3 \\
 0 & 0 & \dots & 0 & 0 \\
 \vdots & \vdots & \ddots & \vdots & \vdots \\
 0 & 0 & \dots & 0 & d_{n-k}
\end{array}\right)\in \mathcal{M}_{n-k,k}(\F),$$
Since $RS=0$ it is easy to see that $V^{k+1}=\left(\begin{array}{cc}
R^{k+1}  & ST^k \\
O & T^{k+1}
\end{array}\right).$ But $T$ is a $k$ by $k$ nilpotent, and therefore $V^{k+1}=0.$ Note that the first column of $V$ is $(0,\dots,0).$
In order to see that $D-V$ is sum of $m$ idempotents we observe that it has the form
$$D-V=\left(\begin{array}{cc}
X  & Y \\
Z & K
\end{array}\right),$$
with $$X=\left(\begin{array}{cccccccc}
m  & 0 & 0 & 0 & \dots & 0 & 0 & 0 \\
1  & 0 & 0 & 0 & \dots & 0 & 0 & 0 \\
0  & 0 & m & 0 & \dots & 0 & 0 & 0 \\
0  & 0 & 1 & 0 & \dots & 0 & 0 & 0 \\
\vdots & \vdots & \vdots & \vdots & \ddots & \vdots & \vdots & \vdots \\
0  & 0 & 0 & 0 & \dots & m & 0 & 0 \\
0  & 0 & 0 & 0 & \dots & 1 & 0 & 0 \\
0  & 0 & 0 & 0 & \dots & 0 & 0 & m
\end{array}\right)\in \mathcal{M}_{n-k}(\F),$$
$$Y=\left(\begin{array}{cccccc}
 0 & 0 & 0 & \dots & 0 & 0 \\
 0 & 0 & 0 & \dots & 0 & d_2 \\
 0 & 0 & 0 & \dots & 0 & 0 \\
 0 & 0 & 0 & \dots & 0 & d_4\\
 \vdots & \vdots & \vdots & \ddots & \vdots & \vdots \\
 0 & 0 & 0 & \dots & 0 & d_{n-k-1} \\
 0 & 0 & 0 & 0 & 0 & 0
\end{array}\right)\in \mathcal{M}_{n-k,k}(\F),$$
$$Z=\left(\begin{array}{ccccc}
 0 & 0 & \dots & 0 & 1 \\
 0 & 0 & \dots & 0 & 0 \\
 \vdots & \vdots & \ddots & \vdots & \vdots \\
 0 & 0 & \dots & 0 & 0 \\
 0 & 0 & \dots & 0 & 0
\end{array}\right)\in \mathcal{M}_{k,n-k}(\F),$$
and $$K=Q-T=mI_k-C_{d_{n-k+1},\dots,d_n}-T=$$ $$mI_k-(F_1+F_2+\dots+F_{m-1}+\left(\begin{array}{cc}
I_{l_1}  & L'_{12} \\
0 & 0
\end{array}\right)-T)-T=$$$$I_k-F_1+\dots+I_k-F_{m-1}+\left(\begin{array}{cc}
0  & -L'_{12} \\
0 & I_{l_1-k}
\end{array}\right)\in \mathcal{M}_{k}(\F).$$
Take $$E_i=\begin{pmatrix}
  \begin{array}{cccccc}
    1 & 0 & 0 & \dots & 0 & 0 \\
    0 & 0 & 0 & \dots & 0 & 0 \\
    0 & 0 & 1 & \dots & 0 & 0 \\
    \vdots & \vdots & \vdots & \vdots & \vdots & \vdots\\
    0 & 0 & 0 & \dots & 0 & 0 \\
    0 & 0 & 0 & \dots & 0 & 1 \
  \end{array}
  & \rvline &
   \begin{matrix}
   \large{0}
  \end{matrix}
   \\
\hline
   \begin{matrix}
   \large{0}
  \end{matrix}
  & \rvline &
  \begin{matrix}
   \large{I_k-F_i}
  \end{matrix}
\end{pmatrix}
,$$ for $i\in\{1,\dots,m-1\},$
and $$E_m=\begin{pmatrix}
  \begin{array}{ccccccc}
    1 & 0 & 0 & \dots & 0 & 0 & 0\\
    1 & 0 & 0 & \dots & 0 & 0 & 0\\
    0 & 0 & 1 & \dots & 0 & 0 & 0\\
    \vdots & \vdots & \vdots & \vdots & \vdots & \vdots\\
    0 & 0 & 0 & \dots & 1 & 0 & 0\\
    0 & 0 & 0 & \dots & 1 & 0 & 0\\
    0 & 0 & 0 & \dots & 0 & 0 & 1
  \end{array}
  & \rvline &
  \begin{array}{cccc}
    0  & \dots & 0 & 0 \\
    0  & \dots & 0 & d_2 \\
    0  & \dots & 0 & 0 \\
    \vdots & \vdots & \vdots & \vdots\\
    0  & \dots & 0 & 0 \\
    0  & \dots & 0 & d_{n-k-1} \\
    0  & \dots & 0 & 0
\end{array}
   \\
\hline
  \begin{array}{ccccccc}
    0 & 0 & 0  & \dots & 0 & 0 & 1 \\
    \vdots & \vdots & \vdots & \vdots & \vdots & \vdots\\
    0 & 0 & 0  & \dots & 0 & 0 & 0
  \end{array}
  & \rvline &
  \begin{array}{cc}
   0 & -L'_{12} \\
   0 &  I_{k-l_1}
  \end{array}
\end{pmatrix}
$$ has the last line $(0,\dots,0,1)$ and is of the form of the idempotent $D-V$ from Lemma $2$ in \cite{BrMe}.
For us $$D-V=E_1+\dots+E_m,$$ and in conclusion the Lemma is proved.
\end{proof}

\begin{rem}\label{formQ}
The transition matrix from the canonical basis of $F^n,$ to the basis constructed in Lemma \ref{mainLem} is upper triangular, with each entry on the main diagonal being equal to $1,$ or $-1.$
\end{rem}

\begin{lem}\label{k=1}
Let $n$ be a positive integer and $m\geq 3$ an integer. Suppose $\F$ is a field and $C\in \M_n(\F)$ is a companion matrix, with $\Tr(C)=\frac{n}{2}\cdot(m-1)+l,$ for some integer $l\in \{1,\dots,m-1\}.$ Then $C$ is similar to a matrix $D$ which has a decomposition $D=E_1+E_2+\dots+E_m+V,$ such that $E_i^2=E_i,$ $i\in\{1,2,\dots,m\},$ $V^{2}=0,$ the first column of $V$ is $(0,0,\ldots,0),$ and the last line of $E_m$ is $(0,0,\ldots,0,1).$
\end{lem}
\begin{proof}
By Lemma \ref{use1}, $C$ is similar with a matrix of the form
$$D=\left(\begin{array}{ccccccc}
    m-1 & 0 & 0 & \dots & 0 & 0 & -d_0\\
    1 & 0 & 0 & \dots & 0 & 0 & -d_1\\
    0 & 1 & m-1 & \dots & 0 & 0 & -d_2\\
    0 & 0 & 1  & \dots & 0 & 0 & -d_3\\
    \vdots & \vdots & \vdots & \vdots & \vdots & \vdots & \vdots\\
    0 & 0 & 0 & \dots & 0 & 0 & -d_{n-3}\\
    0 & 0 & 0 & \dots & 1 & m-1 & -d_{n-2}\\
    0 & 0 & 0 & \dots & 0 & 1 & -d_{n-1}
  \end{array}\right)$$
Since $\Tr(D)=\Tr(C),$ we derive that $-d_{n-1}=l\in\{1,\dots,m-1\}.$ It follows that there exist $e_{i}\in\{0,1\}$, for every $i\in \{1,\dots,m-2\}$ such that $$-d_{n-1}=e_1+\dots+e_{m-2}+0+1.$$
\\ Now consider the following $n\times n$ matrices
\begin{itemize}
\item the idempotents $E_i=\diag(1,0,1,\dots,0,1,e_i),$ for every $i\in \{1,\dots,m-3\};$
\item the idempotent $E_{m-2}=\diag(1,0,1,\dots,1,0,0,e_{m-2});$
\item the idempotent $E_{m-1}=\begin{pmatrix}
  \begin{matrix}
    \large{0}
  \end{matrix}
  & \rvline &
 \begin{matrix}
   \large{0}
 \end{matrix}
 \\
\hline
  \begin{matrix}
   \large{0}
  \end{matrix}
  & \rvline &
  \begin{array}{cc}
   1 & 0 \\
   1 & 0
  \end{array}
\end{pmatrix}$

\item the matrix $$E_m=\begin{pmatrix}
  \begin{array}{ccccccc}
    1 & 0 & 0 & \dots & 0 & 0 & 0\\
    1 & 0 & 0 & \dots & 0 & 0 & 0\\
    0 & 0 & 1 & \dots & 0 & 0 & 0\\
    \vdots & \vdots & \vdots & \vdots & \vdots & \vdots\\
    0 & 0 & 0 & \dots & 1 & 0 & 0\\
    0 & 0 & 0 & \dots & 1 & 0 & 0\\
    0 & 0 & 0 & \dots & 0 & 0 & 1
  \end{array}
  & \rvline &
  \begin{array}{c}
    0 \\
    -d_1 \\
    0 \\
    -d_2 \\
    \vdots \\
    -d_{n-3} \\
    0
\end{array}
   \\
\hline
  \begin{array}{ccccccc}
    0 & 0 & 0  & \dots & 0 & 0 & 0
  \end{array}
  & \rvline &
  \begin{array}{c}
   1
  \end{array}
\end{pmatrix} =\left(\begin{array}{cc}
                                          X & Y \\
                                          0 & 1 \end{array}\right),$$
      which is idempotent since $X$ is idempotent, and $XY=0;$ moreover it has last line $(0,\dots,0,1);$
\item the matrix $$V=\begin{pmatrix}
  \begin{array}{ccccccc}
    0 & 0 & 0 & \dots & 0 & 0 & 0\\
    0 & 0 & 0 & \dots & 0 & 0 & 0\\
    0 & 1 & 0 & \dots & 0 & 0 & 0\\
    \vdots & \vdots & \vdots & \vdots & \vdots & \vdots\\
    0 & 0 & 0 & \dots & 0 & 0 & 0\\
    0 & 0 & 0 & \dots & 0 & 0 & 0\\
    0 & 0 & 0 & \dots & 0 & 1 & 0
  \end{array}
  & \rvline &
  \begin{array}{c}
    -d_0 \\
    0 \\
    -d_2 \\
    0 \\
    \vdots \\
    0 \\
    -d_{n-1}
\end{array}
   \\
\hline
  \begin{array}{ccccccc}
    0 & 0 & 0  & \dots & 0 & 0 & 0
  \end{array}
  & \rvline &
  \begin{array}{c}
   0
  \end{array}
\end{pmatrix} =\left(\begin{array}{cc}
                                          R & S \\
                                          0 & 0 \end{array}\right),$$
      which is such that $V^2=0,$ since $R^2=0$, and $RS=0;$ moreover it has the first column $(0,\dots,0).$
\end{itemize}
We have the desired decomposition $$D=E_1+E_2+\dots+E_m+V,$$ and the conclusion is proved.
\end{proof}

\section{$m$-nil-clean nonderogatory matrices}

\begin{thm}\label{mainThm}
Let $\F$ be a field of odd characteristic $p.$ Let $m\geq 2,$ and $n$ be positive integers, such that $mn-1\geq p\geq n.$ Suppose $A\in \M_n(\F)$ is a nonderogatory matrix such that $\Tr(A)=t\cdot 1,$ $t\in \{0,\dots,p-1\}.$ There exist the idempotents $E_i,$ for every $i\in \{1,\dots,m\},$ and the nilpotent $V,$ such that $A=E_1+\dots+E_m+V,$ and $V^{[\frac{p-2}{m}]+2}=0.$
\end{thm}
\begin{proof}
We can consider only decomposition for $C,$ a companion matrix such that $A\sim C,$ since a matrix similar to an idempotent is an idempotent, and a matrix similar to a nilpotent is nilpotent.
Let $k=[\frac{p-2}{m}]+1.$ Then $k>\frac{p-2}{m},$ and hence $km-1\geq p-2.$
\begin{itemize}
   \item $k\geq n-1;$ $C$ is $m$-nil clean with the involved nilpotent $V,$ such that $V^n=0,$ therefore $V^{k+1}=0.$
   \item $k\leq n-2;$
         \begin{itemize}
         \item $n-k$ is odd; For further use we are going to prove more, there exists $D$ such that $C\sim D,$ and $D$ has an $m$-nil clean decomposition, with properties like in Lemma \ref{mainLem}.
               \begin{itemize}
               \item $k=1;$ we have $m-1\geq p-2,$ and also $n>3,$ and therefore $p\geq 5.$
                     \begin{itemize}
                     \item $m-1\geq p;$ then $\Tr(C)=\frac{n}{2}\cdot(m-1)+l,$ for some $l\in \{1,\dots,m-1\}.$ By Lemma \ref{k=1}, we get that $C$ has a desired $m$-nil-clean decomposition.
                     \item $m-1 \in \{p-1,p-2\};$ that is $m\in \{p,p-1\};$ hence $m\geq 4.$

                           By Lemma \ref{use1}, for every $r\in \{1,\dots,n\},$ there exists a companion matrix $C'(r),$ such that $$C\sim \diag (\underbrace{1,\dots,1 }_{r \textrm{-times}} ,0,\dots,0) +C'(r).$$ Hence
                         $$\Tr(C'(r))=(\frac{n}{2}\cdot (m-1)+(l-r))\cdot 1=$$$$(\frac{n}{2}\cdot(m-2)+(l-r+\frac{n}{2})\cdot 1.$$
                         Let $$l'(r)=(l-r+\frac{n}{2})\cdot1.$$
                         Assume that for every $r\in \{1,\dots, n\},$ we have $l'(r)\in \{p-2,p-1,p\}.$ Note that since $n\leq p,$ we derive that $l'(r_1)=l'(r_2),$ if and only if $r_1=r_2.$ Therefore $n\leq 3,$ contradiction.

                          Hence there exists $r\in \{1,\dots,n\},$ such that $l'(r)\in \{1,\dots,p-3\}.$ Consequently $l'(r)\in \{1,\dots,m-2\}.$ By Lemma \ref{k=1}, $C'(r)\sim D'$, and $D'$ has a desired $(m-1)$-nil-clean decomposition. Moreover $C'=PD'P^{-1},$ and we know by Remark \ref{formP} that $P$ is upper triangular, with main diagonal $(1,\dots,1,1)$. Therefore the product of $P$, an idempotent with last line $(0,\dots,0,1),$ and $P^{-1}$ has last line $(0,\dots,0,1);$ also the product of $P$, a nilpotent with first column $(0,\dots,0),$ and $P^{-1}$ has first column $(0,\dots,0).$ Hence, since $C\sim \diag (\underbrace{1,\dots,1 }_{r \textrm{-times}} ,0,\dots,0)+PD'P^{-1},$ we obtain that $C$ has a desired $m$-nil-clean decomposition.

                     \end{itemize}
               \item $k\geq 2;$ we have the following situations:
                     \begin{itemize}
                     \item $km-1\geq p;$ then $\Tr(C)=(\frac{n+k+1}{2}\cdot m-l)\cdot 1,$ for some $l\in \{1,\dots,km-1\}.$ Therefore by Lemma \ref{mainLem}, $C$ has an $m$-nil-clean desired decomposition.
                     \item $km-1=p-1;$ then $km=p,$ and since $k\geq 2,$ it follows the contradiction $m=1.$
                     \item $km-1=p-2;$ then $km=p-1.$ By Lemma \ref{use1}, for every $r\in \{1,\dots,n\},$ there exists a companion matrix $C'(r),$ such that $$C\sim \diag (\underbrace{1,\dots,1 }_{r \textrm{-times}} ,0,\dots,0) +C'(r).$$ Hence
                         $$\Tr(C'(r))=(\frac{n+k+1}{2}\cdot m-(l+r))\cdot 1=$$$$(\frac{n+k+1}{2}\cdot(m-1)-(l+r-\frac{n+k+1}{2}))\cdot 1.$$
                         Let $$l'(r)=(l+r-\frac{n+k+1}{2})\cdot1.$$
                         Assume that for every $r\in \{1,\dots, n\},$ we have $l'(r)\in \{k(m-1),\dots,p\}.$ Note that since $n\leq p,$ we derive that $l'(r_1)=l'(r_2),$ if and only if $r_1=r_2.$ Therefore $n\leq p-(k(m-1)-1)=p-(p-1-k-1)=k+2\leq n$. Consequently $k+2=n,$ and the contradiction $n-k$ is even.

                         Therefore there exists $r\in \{1,\dots,n\},$ such that $l'(r)\in \{1,\dots,k(m-1)-1\}.$ By Lemma \ref{mainLem}, $C'(r)\sim D'$, and $D'$ has a desired $(m-1)$-nil-clean decomposition. Moreover $C'(r)=QD'Q^{-1},$ and we know by Remark \ref{formQ} that $Q$ is upper triangular, with each entry on main diagonal being equal to $1$ or $-1.$ Therefore the product of $Q$, an idempotent with last line $(0,\dots,0,1),$ and $Q^{-1}$ has last line $(0,\dots,0,1);$ also the product of $Q$, a nilpotent with first column $(0,\dots,0),$ and $Q^{-1}$ has first column $(0,\dots,0)$ Therefore, since $C\sim \diag (\underbrace{1,\dots,1 }_{r \textrm{-times}} ,0,\dots,0)+QDQ^{-1},$ we obtain that $C$ has a desired $m$-nil-clean decomposition.

                     \end{itemize}
               \end{itemize}
         \item $n-k$ is even
           We are going to use a technique similar with the one used for nil-clean decompositions in the proof of the main theorem from \cite{BrMe},
      obtaining a decomposition as sum of $m$ idempotents and a nilpotent with nilpotence index at most $k+1$, for the companion matrix $C\in \mathcal{M}_n(\F),$  in the following way:
\end{itemize}
     we know that $C$ is similar to a matrix of the form

\[D'=
\begin{pmatrix}
  \begin{array}{c}
    0
  \end{array}
  & \rvline &
  \begin{array}{cccc}
    0  & \dots & 0 & d_1
\end{array}
   \\
\hline
  \begin{array}{c}
    1 \\
    0 \\
    \vdots\\
    0
  \end{array}
  & \rvline &
  \begin{matrix}
   \large{D}
  \end{matrix}
\end{pmatrix}
\]
where $D\in \mathcal{M}_{3}(\F)$ is the nonderogatory matrix $D$ we have found in previous case, where we proved that there exists an $m$-nil clean   decomposition for $D=E_1+\dots+E_m+V,$ last line of $E_m$ is $(0,\dots,0,1),$ first column of $V$ is $(0,\dots,0),$ and $V^{k+1}=0.$
 Then since $k\leq n-1,$ $D'=E_1'+\dots+E_m'+W,$ with $E_i'=\left(\begin{array}{cc}
                                                             0 & 0 \\
                                                             0 & E_i\end{array}\right)$, for every $i\in \{1,2,\dots,m-1\},$
$E'_m=
\begin{pmatrix}
  \begin{array}{c}
    0
  \end{array}
  & \rvline &
  \begin{array}{cccc}
    0  & \dots & 0 & d_1
\end{array}
   \\
\hline
  \begin{array}{c}
    0 \\
    0 \\
    \vdots\\
    0
  \end{array}
  & \rvline &
  \begin{matrix}
   \large{E_m}
  \end{matrix}
\end{pmatrix}
,$ and
\[W=
\begin{pmatrix}
  \begin{array}{c}
    0
  \end{array}
  & \rvline &
 \begin{matrix}
   0 & \dots & 0 & 0
 \end{matrix}
 \\
\hline
  \begin{array}{c}
     1 \\
     0 \\
     \vdots\\
     0 \\
  \end{array}
  & \rvline &
  \begin{matrix}
   \large{V}
  \end{matrix}
\end{pmatrix}
\]
Since last line of $E_m$ is $(0,\dots,0,1),$ we derive that $E'_m$ is idempotent.
Since first column of $V$ is $(0,\dots,0)$ we have $W^{k+1}=0$. Therefore since $D'\sim C$, a matrix similar to an idempotent is a idempotent, and a matrix similar to a nilpotent is nilpotent, we have found a decomposition as sum of $m$ idempotents and a nilpotent with nilpotence index at most $k+1$ for $C.$
\end{itemize}
\end{proof}

The following Proposition proves that in general we cannot decrease the bound
of the nilpotence degrees of the nilpotent matrices which can be used to obtain
decompositions as in Theorem \ref{mainThm}.

\begin{prop}
Let $\F$ a field of odd characteristic $p,$ and $n=[\frac{p-2}{2}]+2.$ Suppose $C\in \M_n(\F)$ is a companion matrix, with $\Tr(C)=(2n-1)\cdot 1,$ which does not have $0$ or $1$ as eigenvalue. Assume $C=E_1+E_2+V,$ such that $E_1^2=E_1,$ $E_2^2=E_2,$ and $V$ is nilpotent. Then $V^{[\frac{p-2}{2}]+1}\neq 0.$
\end{prop}
\begin{proof}
There exist $k_1,$ and $k_2$ such that $$E_1 \sim \diag (\underbrace{1,\dots,1 }_{k_1 \textrm{-times}} ,0,\dots,0),$$ and $$E_2\sim \diag (\underbrace{1,\dots,1 }_{k_2 \textrm{-times}} ,0,\dots,0).$$ Then $\Tr(E_1)=(rank E_1)\cdot 1=k_1\cdot 1,$ and
$\Tr(E_2)=(rank E_2)\cdot 1=k_2\cdot 1.$
Since $\Tr(V)=0,$ we obtain the following
$$1\cdot 1=(2n-1)\cdot 1=\Tr(C)=(k_1+k_2)\cdot 1.$$
We derive that $\{k_1,k_2\}=\{n-1,n\},$ or $\{k_1,k_2\}=\{0,1\}.$

 Assume $V^{n-1}=0,$ then $V$ has at least two blocks in its Frobenius normal form, and therefore at least two zero lines in this form. We derive that $rk(V)\leq n-2.$

Since $0$ is not an eigenvalue for $C,$ we obtain that
$$n=rk(C)\leq rk(E_1)+rk(E_2)+rk(V).$$
Hence $rk(E_1)+rk(E_2)\geq 2,$ and therefore $\{k_1,k_2\}\neq\{0,1\}.$

  Moreover $2$ is not an eigenvalue for $C,$ and we have:
$$n=rk(2I_n-C)\leq rk(I_n-E_1)+rk(I_n-E_2)+rk(-V).$$
Hence $rk(I_n-E_1)+rk(I_n-E_2)\geq 2,$ and therefore $\{k_1,k_2\}\neq\{n-1,n\}.$

In conclusion $V^{n-1}\neq 0.$
\end{proof}

\end{document}